\documentclass[a4paper,11pt,final,leqno,notitlepage]{article}

\usepackage{amsmath}
\usepackage{amsfonts}
\usepackage{amssymb}
\usepackage{amsthm}
\usepackage{mathtools} 
\usepackage{dsfont}
\usepackage{stackengine}
\usepackage{kpfonts}
\usepackage{mathtools}
\stackMath

\usepackage{geometry}
\usepackage[blocks]{authblk}

\usepackage[T1]{fontenc}

\usepackage{fullpage}

\usepackage{enumerate}
\usepackage{enumitem}

\begin{document}
\renewcommand{\thepage}{\small\arabic{page}}
\renewcommand{\thefootnote}{(\arabic{footnote})}
\renewcommand{\thesection}{\arabic{chapter}.\arabic{section}}
\renewcommand{\thesubsection}{\arabic{subsection}}

\renewcommand\Affilfont{\small}

\author{Maria Trybu\l{}a}
\affil{Institute of Mathematics and Informatics\\Bulgarian Academy of Sciences\\ Akad. Georgi Bonchev 8, 1113 Sofia, Bulgaria\\ trybula@math.bas.bg}

\affil{Faculty of Mathematics and Informatics\\Adam Mickiewicz University in Pozna\'{n} \\ Uniwersytetu Pozna\'{n}skiego 4, 61-614 Pozna\'{n}, Poland\\ maria.h.trybula@gmail.com}

\title{Representation of Operators on Spaces of Holomorphic Functions in $\mathbb{C}^n$}
\maketitle

\begin{abstract}
We investigate operators between spaces of holomorphic functions in several complex variables. Let $G_1, G_2 \subset \mathbb{C}^n$ be cylindrical domains. We construct a canonical map from the space of bounded linear operators $\mathcal{L}(H(G_1), H(G_2))$ to $H(G_1^b \times G_2)$ and prove that it is a topological isomorphism (Theorem~\ref{pierwsze twierdzenie}). We then establish uniform estimates for operators on bounded, complete $n$-circled domains (Theorem~\ref{thm:4.8}) and show that sequences of operators on smaller domains satisfying suitable uniform bounds uniquely determine a global operator (Theorem~\ref{thm:4.9}). Together, these results provide a unified framework for representing and extending operators on spaces of holomorphic functions in several complex variables.

 \footnotetext[1]{


The research was supported by National Science Fund (Bulgaria), grant no. KP-06-N82/6.}

\end{abstract}
\newtheorem{thm}{Theorem}[subsection]
\newtheorem{lemma}[thm]{Lemma}
\newtheorem{proposition}[thm]{Proposition}
\newtheorem{corollary}[thm]{Corollary}
\newtheorem{remark}[thm]{Remark}
\newtheorem*{example}{Example}

\subsection{Introduction}

In his seminal work, K\"{o}the \cite{Kothe} provided a complete description of the dual space of holomorphic functions on a domain $G \subset \mathbb{C}$, establishing a framework for representing bounded linear operators between spaces of holomorphic functions on planar domains. These discoveries have since become essential tools in various areas, including the study of Euler differential operators and Hadamard multiplication operators.

The first significant extension to several complex variables was given by Tillmann \cite{Tillmann}. While the complexity of domains in $\mathbb{C}^n$ and the absence of effective integral representation formulas limited the generality of his results, he established an important duality theorem for spaces of holomorphic functions on the Cartesian products of planar domains and their continuous linear functionals.

In this paper, we investigate bounded linear operators between spaces of holomorphic functions in several complex variables. Let $G_1, G_2 \subset \mathbb{C}^n$ be cylindrical domains. We construct a canonical map from $\mathcal{L}(H(G_1), H(G_2))$ to $H(G_1^b \times G_2)$ and show that it is a topological isomorphism of Fr\'{e}chet spaces, providing a concrete representation of bounded operators via holomorphic kernels (Theorem \ref{pierwsze twierdzenie}). We then establish uniform estimates for operators on bounded, complete $n$-circled domains, proving that the composition with canonical projections extends uniquely to bounded operators on the corresponding smaller domains (Theorem \ref{thm:4.8}), and we show that a sequence of bounded operators on these smaller domains satisfying suitable uniform bounds determines a unique global bounded operator whose restrictions coincide with the sequence (Theorem \ref{thm:4.9}). Together, these results provide a unified and coherent framework for the representation of bounded linear operators between spaces of holomorphic functions in several complex variables, generalizing classical results of K\"{o}the and offering new tools for operator-theoretic applications in complex analysis.
  
Notation and preliminary facts are collected in Section 2.

\subsection{Preliminaries}

We begin by fixing notation. The symbols $\mathbb{N}$, $\mathbb{R}$, $\mathbb{R}_{>0}$, $\mathbb{C}$, and $\hat{\mathbb{C}}$ denote, respectively, the natural numbers (excluding zero), the real numbers, the positive real numbers, the complex plane, and the Riemann sphere. For $A\subset\mathbb{C}$, its complement in $\hat{\mathbb{C}}$ is denoted $A^c$. For $n,j\in\mathbb{N}$, we write $\mathbf{1}_n:=(1,\ldots,1)$ ($n$ times) and $e_j:=(0,\ldots,0,1,0,\ldots,0)\in\mathbb{C}^n$ (with $1$ in the $j$th coordinate). The Euclidean norm of $z\in\mathbb{C}^n$ is $|z|$.

For $a\in\mathbb{C}$ and $r>0$, let $\mathbb{D}(a,r)$ denote the open disc with center $a$ and radius $r$; we write $\mathbb{D}(r)$ if $a=0$. For $a\in\mathbb{C}^n$ and $r\in\mathbb{R}_{>0}^n$, the polydisc with center $a$ and polyradius $r$ is
\[
\mathbb{P}_n(a,r) := \mathbb{D}(a_1,r_1)\times\cdots\times\mathbb{D}(a_n,r_n),
\]
and $\mathbb{P}_n(r)$ if $a=0$.

If $\Omega=\Omega_1\times\cdots\times\Omega_n\subset\mathbb{C}^n$, the set $\partial\Omega_1\times\cdots\times\partial\Omega_n$ is called the distinguished boundary of $\Omega$ and denoted $\partial_0\Omega$. If each $\partial\Omega_j$ is piecewise $\mathcal{C}^1$ and $f$ is continuous on $\partial_0\Omega$, we define
\[
\int_{\partial_0\Omega} \frac{f(\zeta)}{\zeta-z} \, d\zeta
:= \int_{\partial\Omega_1} \cdots \int_{\partial\Omega_n} 
\frac{f(\zeta_1,\ldots,\zeta_n)}{(\zeta_1-z_1)\cdots(\zeta_n-z_n)} \, 
d\zeta_1\cdots d\zeta_n, \quad z\in\Omega.
\]

For $a\in\mathbb{C}$ and $z=(z_1,\ldots,z_n)\in\mathbb{C}^n$, define $az := (az_1,\ldots,az_n)$. For $A\subset\mathbb{C}^n$, set $aA := \{az : z\in A\}$. We also use coordinatewise addition: for $z,w\in\mathbb{C}^n$, $z+w := (z_1+w_1,\ldots,z_n+w_n)$.

Let $\alpha\in (\mathbb{N}\cup\{0\})^n$ be a multi-index. We write
\[
\alpha! := \alpha_1!\cdots \alpha_n!, \qquad
|\alpha| := \alpha_1+\cdots+\alpha_n, \qquad
\zeta^\alpha := \zeta_1^{\alpha_1}\cdots \zeta_n^{\alpha_n}.
\]

For a compact set $K\subset\mathbb{C}^n$ and a bounded function $f:K\to\mathbb{C}$, the supremum norm is $\lVert f\rVert_K := \sup_{z\in K} |f(z)|$, and $\mathcal{C}(K)$ denotes the space of continuous functions on $K$ with this norm.

For an open set $\Omega\subset\hat{\mathbb{C}}^n$, $H(\Omega)$ denotes the space of holomorphic functions on $\Omega$, equipped with the topology of uniform convergence on compact sets. If $\Omega\not\subset \mathbb{C}^n$, functions in $H(\Omega)$ are assumed to vanish at any infinite point. For $S\subset\hat{\mathbb{C}}^n$, set
\[
H(S) := \bigcup_{\omega\supset S \text{ open}} H(\omega),
\]
the space of germs of holomorphic functions on $S$, endowed with the finest locally convex topology making all restriction maps $H(\omega)\to H(S)$ continuous.

Finally, $B(G)$ denotes the space of functions holomorphic on $G$ and extending continuously to $\partial G$. Equipped with the norm $\lVert\cdot\rVert_D$, $B(G)$ is a Banach space.

If $X$ and $Y$ are Fréchet spaces, $\mathcal{L}(X,Y)$ denotes the space of continuous linear operators from $X$ to $Y$. 

For further background on Functional Analysis and Complex Analysis not treated in this paper, we refer to \cite{Vogt} and \cite{Jarnicki}.

\subsection{Special case}
Denote by $\mathcal{G}_n$ the family of all cylindrical domains in $\mathbb{C}^n$ whose components are of finite connectivity and have closures in $\hat{\mathbb{C}}$  that are not equal to $\hat{\mathbb{C}}$ itself. The subfamily of domains in $\mathcal{G}_n$ in which every component has analytic boundary will be denoted by $\mathcal{G}_n^\omega$.
\begin{proposition}
If $G\in\mathcal{G}_n^\omega$, then $H(\overline{G})$ is dense in $B(G)$.
\end{proposition} 
\begin{proof}
Let $G=G_1\times\ldots\times G_n$. Fix $f\in B(G)$. Before reaching the conclusion, let us consider a special case. 
Assume that every component $G_j$ is simply connected. For every $j$ let $\varphi_j:G_j\rightarrow\mathbb{D}$ denote a Riemann mapping. Since $\partial G_j$ is analytic, the mapping $\varphi_j$ is holomorphic and univalent on some neighborhood of $\overline{G}_j$ (see \cite[Proposition 3.1]{Pommerenke}). Then , for $r \in (0,1)$, the function $f_r$ defined by 
$$f_r(z)=f_r(z_1,\ldots,z_n):=f\left(\varphi_1^{-1}(r\varphi_1(z_1)),\ldots,\varphi^{-1}_n(r\varphi_n(z_n))\right)$$ 
is well-defined and holomorphic in a neighborhood of $\overline{G}$. This construction yields an approximation of $f$ that meets the required conditions.
\newline Next, let us consider the general case. By the Cauchy Integral Formula:
$$f(z)=\frac{1}{(2\pi i)^n}\int_{\partial_0 G}\frac{f(\zeta)}{\zeta-z}d\zeta =\sum_{m\in\mathbb{N}^n,\,1\leq m_j\leq k_j}\frac{1}{(2\pi i)^n}\int_{\partial G_{n,{m_n}}}\ldots\int_{\partial G_{1,{m_1}}}\frac{f(\zeta)}{\zeta-z}d\zeta_1\ldots d\zeta_n,$$
for $z\in G$, where $\partial G_j=\cup_{l=1}^{k_j} G_{j,{l}},\,j=1,\ldots,n$. For $m=(m_1,\ldots,m_n)\in\mathbb{N}^n$ such that $1\leq m_j\leq k_j,\,j=1,\ldots,n$, we define 
$$f_m(z):=\frac{1}{(2\pi i)^n}\int_{\partial G_{n,{m_n}}}\ldots\int_{\partial G_{1,{m_1}}}\frac{f(\zeta)}{\zeta-z}d\zeta_1\ldots d\zeta_n,\ \ z\notin \partial_0 G.$$
Observe that $f_m$ is holomorphic in the domain $G(m)$, where $G(m)$ denotes the unique domain bounded by $\partial_0 G_{1,{m_1}}\times\ldots\times \partial_0 G_{n,{m_n}}$ that contains $G$. Moreover, $f_m$ has a continuous extension to $\partial_0 G_{1,{m_1}}\times\ldots\times \partial_0 G_{n,{m_n}}$. This shows that $B(G)$ is a direct sum of $B(G(m))$'s. Now, for every $f_m$ we can apply the special case considered earlier.  
\end{proof}

For a cylindrical set $E=E_1\times\ldots\times E_n\subset\mathbb{C}^n$, we define a complementary cylinder set $E^b$ as follows
$$E^b:=E_1^c\times\ldots\times E_n^c.$$

\begin{thm}\textup{(\cite[Satz 1]{Tillmann})}\label{Satz Tillmann} Let $G\in\mathcal{G}_n$. Spaces $H(G)$ and $H(G^b)$ are strongly dual. 
\end{thm}

\begin{proposition}\label{ważne}
Let $G\in\mathcal{G}_n$. For $u\in B(G)'$ the function $\widetilde{u}$ defined by
$$\widetilde{u}(\lambda):=u\left( \frac{1}{z-\lambda}\right)=u\left( \frac{1}{z_1-\lambda_1}\cdot\ldots\cdot \frac{1}{z_n-\lambda_n}\right),\ \ \lambda=(\lambda_1,\ldots,\lambda_n)\in \overline{G}^{b},$$
is holomorphic on $\overline{G}^{b}$. Moreover, the action of $u$ on $H(\overline{G})$ is given by the formula
$$u(f)=\frac{(-1)^n}{(2\pi i)^n}\int_{\partial_0 G}\widetilde{u}(\lambda) f(\lambda)d\lambda,\ \ f\in H(\overline{G}).$$
\end{proposition}

\begin{proof}
The proof is similar to that of Theorem \ref{Satz Tillmann}. For the sake of completeness, we include it here. By the Hartogs Theorem, for the first part of the statement, it is enough to show that $\widetilde{u}$ is separately holomorphic. For $\lambda \in \overline{G}^{b}$ and sufficiently small $h$, we have:
\begin{multline*}
\frac{1}{h}\left[ \widetilde{u}(\lambda_1,\ldots,\lambda_j +h,\ldots,\lambda_n)-\widetilde{u}(\lambda_1,\ldots,\lambda_j ,\ldots,\lambda_n)\right] = \\ 
u\left(\frac{1}{h}\left(\frac{1}{z_j-\lambda_j-h}-\frac{1}{z_j-\lambda_j}\right)\prod_{k\not=j} \frac{1}{z_k-\lambda_k}\right) = u\left(\frac{1}{(z_j-\lambda_j-h)(z_j-\lambda_j)}\prod_{k\not=j} \frac{1}{z_k-\lambda_k}\right).
\end{multline*}
Since $\frac{1}{(z_j-\lambda_j-h)(z_j-\lambda_j)}$ tends to $\frac{1}{(z_j-\lambda_j)^2}$ in $B(G)$ as $h$ approaches $0$, we obtain that 
$$
\frac{\partial \widetilde{u}(\lambda)}{\partial \lambda_j}=u\left(\frac{1}{(z-\lambda)^{\mathbf{1}_n+e_j}}\right).
$$
By a similar argument, we also conclude that 
$$\widetilde{u}(\lambda)\to 0\ \textup{ as }\ \lvert \lambda\rvert\to\infty.$$

For the second part of the statement, let us fix $f\in H(V)$, where $V\in\mathcal{G}_n$ is such that $\overline{G}\subset V$. Let $U\in\mathcal{G}^\omega_n$ be such that $\overline{G}\subset U\subset\overline{U}\subset V$. By the Cauchy Integral Formula, we may write:
\begin{equation}\label{wniosek 1}
f(z)=\frac{1}{(2\pi i)^n}\int_{\partial_0 U}\frac{f(\lambda)}{\lambda-z}d\lambda,\ \ \ z\in U.
\end{equation}
Using Riemann sums to approximate (\ref{wniosek 1}), we get
$$
u(f)=\frac{(-1)^n}{(2\pi i)^n}\int_{\partial_0 U} f(\lambda)\widetilde{u}(\lambda)d\lambda.
$$
\end{proof}
\begin{remark}
The converse in Proposition \ref{ważne} is false. To see this, let $G=\mathbb{P}_n$ and define $\widetilde{u}(\lambda)=\frac{1}{(\lambda-1)^{\mathbf{1}_n+\mathbf{1}_n}}$ for $\lambda\in\mathbb{P}_n$. Suppose that $\widetilde{u}$ corresponds to some $u\in B(G)'$. Then, we would have the inequality 
$$|\widetilde{u}(\lambda)|\leq \lVert u\rVert \Big{\lVert}\frac{1}{(\lambda-1)^{\mathbf{1}_n}}\Big{ \rVert}_{\mathbb{P}_n}.$$ 
However, it is clear that $\widetilde{u}$ grows faster near $\partial_0\mathbb{P}_n$. This contradiction shows that $\widetilde{u}$ cannot correspond to any $u\in B(G)'$, thereby disproving the converse.
\end{remark}

\begin{corollary}
Let $G=G_1\times\ldots\times G_n\in\mathcal{G}_n^\omega$ and $D=D_1\times\ldots\times D_n\in\mathcal{G}_n$ be domains such that $\overline{G}\subset D$ and every component of $G_j^c$ intersects $D_j^c$ for $j=1,\ldots,n$. Then $H(D)$ is dense in $B(G)$.
\end{corollary}

\begin{proof}
Let $u \in H(D)'$ be fixed. Suppose that $v$ and $w$ are extensions of $u$ to $B(G)$. By Proposition \ref{ważne}, we know that $v$ and $w$ correspond to holomorphic functions on $\overline{G}^b$, denoted $\widetilde{v}$ and $\widetilde{w}$, respectively. These are given by the formulas 
$$\widetilde{v}(\lambda)=v\left(\frac{1}{\cdot-\lambda}\right)\ \  \textup{ and } \ \ \ \widetilde{w}(\lambda)=w\left(\frac{1}{\cdot-\lambda}\right),\ \ \ \lambda \in \overline{G}^b. $$ 
For $\lambda \in D^b$, we have $\widetilde{w}(\lambda)=\widetilde{u}(\lambda)=\widetilde{v}(\lambda)$. 
Since for every $j$ every component of $G_j^c$ intersects  $D_j^c$ nontrivially, the identity principle implies that $\widetilde{w} = \widetilde{v}$. Consequently, $w = v$ on $H(\overline{G})$. Finally, since $H(\overline{G})$ is dense in $B(G)$, this completes the proof.
\end{proof}

We now present the main theorem of this section.

\begin{thm}\label{pierwsze twierdzenie}
Let $G_1, G_2 \in \mathcal{G}_n$. Define the map
\[
\Phi: \mathcal{L}\bigl(H(G_1), H(G_2)\bigr) \longrightarrow H\bigl(G_1^b \times G_2\bigr)
\]
by
\[
\Phi(A)(\zeta, z) := A\Bigl(\frac{1}{w_1-\zeta_1} \cdot \frac{1}{w_2-\zeta_2} \cdots \frac{1}{w_n-\zeta_n}\Bigr)(z),
\quad A \in \mathcal{L}\bigl(H(G_1), H(G_2)\bigr), \; (\zeta,z) \in G_1^b \times G_2.
\]
Then $\Phi$ is a topological isomorphism of Frech\'{e}t spaces.
\end{thm}

\begin{proof}
Assume $G_j=G_1^{(j)}\times\ldots G_n^{(j)}$ with $G_k^{(j)}\subset\mathbb{C}$ being a domain for $j=1,2$ and $k=1\ldots ,n$. For every $j$ and $k$, we choose an exhaustion $\big{(}G^{(j)}_{k,l}\big{)}_{l=0}^{\infty}$ of $G_{k}^{(j)}$ by open and relatively compact domains $G^{(j)}_{k,l}\subset G_{k}^{(j)}$ in such a way that $\overline{G^{(j)}_{k,l}}\subset G^{(j)}_{k,l+1}$ and the boundary of every $G^{(j)}_{k,l}$ consists of finitely many pairwise disjoint, smooth and rectifiable Jordan curves oriented positively with respect to $G_{k,l}^{(j)}$. Moreover, ensure that every component of $\hat{\mathbb{C}}\setminus G_{k,l}^{(j)}$ contains a point of $\hat{\mathbb{C}}\setminus G_k^{(j)}$. For convenience, we define $G_{j,l}:=G^{(j)}_{1,l}\times\ldots\times G^{(j)}_{n,l}$ for $j=1,2$ and $l\in\mathbb{N}$. 

We now fix $A\in\mathcal{L}(H(G_1),H(G_2))$. For every $l\in\mathbb{N}$, there exist $N(l)\in\mathbb{N}$ and $\varepsilon(l)>0$ such that
$$A\big{(}\{f\in H(G_1):\lVert f\rVert_{1,N(l)}\leq\varepsilon(l)\}\big{)}\subset\{g\in H(G_2):\lVert g\rVert_{2,l}\leq 1\},$$
where $\lVert\ \rVert_{k,l}$ denotes the norm $\lVert\ \rVert_{G_{k,l}}$. Since $H(G_1)$ is dense in $B(G_{1,N(l)})$ and $B(G_{2,l})$ is complete, $A$ can be uniquely extended to a continuous mapping $A_l$ on the entire $B(G_{1,N(l)})$ with values in $B(G_{2,l})$.  Consequently, for all $f\in B(G_{1,N(l)})$ the inequality
\begin{equation}\label{M_k}
\lVert A_l f\rVert_{2,l} \leq M_l\rVert f\rVert_{1,N(l)}
\end{equation}
holds, where $M_l=\varepsilon (l)^{-1}$. 
\newline To proceed, we define a function $a_l$ of $2n$ variables by
$$a_l(\zeta,z):=A_l\left( \frac{1}{w_1-\zeta_1}\cdot\ldots\cdot\frac{1}{w_n-\zeta_n}\right)(z),\ \ (\zeta ,z)\in \overline{G}_{1,N(l)}^{\,c}\times \overline{G}_{2,l}.$$ 
For every point $(\zeta,z)$ in $\overline{G}_{1,N(l)}^{\,c}\times \overline{G}_{2,l}$ with at least one coordinate at infinity, we set $a_l(\zeta,z)=0$.  
Fixing $\zeta\in \overline{G}_{1,N(l)}^{\,c}$, the function $a_l(\zeta,\cdot)$ is holomorphic on $G_{2,l}$ and continuous on $\overline{G}_{2,l}$.  To verify the differentiability of $a_l$ with respect to $\zeta$, we now fix $z\in G_{2,l}$. For $\zeta\in\overline{G}^c_{1,N(l)}$ and sufficiently small $h$, we have:
\begin{multline*}\frac{1}{h}\left[ a_l((\zeta_1,\ldots,\zeta_j +h,\ldots,\zeta_n),z)-a_k((\zeta_1,\ldots,\zeta_j ,\ldots,\zeta_n),z)\right] = \\ 
A_l\left(\frac{1}{h}\left(\frac{1}{w_j-\zeta_j-h}-\frac{1}{w_j-\zeta_j}\right)\prod_{k\not=j} \frac{1}{w_k-\zeta_k}\right)(z)=A_l\left(\frac{1}{(w_j-\zeta_j-h)(w_j-\zeta_j)}\prod_{k\not=j} \frac{1}{w_k-\zeta_k}\right)(z).
\end{multline*}
As $h\rightarrow 0$, the quotient $\frac{1}{(w_j-\zeta_j-h)(w_j-\zeta_j)}$ converges to $\frac{1}{(w_j-\zeta_j)^2}$ in $B(G_{1,N(l)})$. Therefore, 
$$\frac{\partial a_l(\zeta,z)}{\partial \zeta_j}=A_l\left(\frac{1}{(w-\zeta)^{\mathbf{1}_n+e_j}}\right)(z),$$ 
where $\mathbf{1}_n+e_j$ denotes the coordinatewise addition of multi-indices.
From (\ref{M_k}), it follows that as $\zeta\rightarrow\infty$, then $a_l$ tends to $0$. Hence, $a_l$ is separately holomorphic  on $\overline{G}_{1,N(l)}^{\,c}\times G_{2,l}$  and vanishes at every $(\zeta,z)$ where at least one coordinate is at infinity. Therefore, by the Hartogs Theorem $a_l$ is holomorphic on $\overline{G}_{1,N(l)}^{\,c}\times G_{2,l}$. 
\newline Furthermore, for $l_1\leq l_2$, the functions $a_{l_1}$ and $a_{l_2}$  agree on $(\overline{G}_{1,N(l_1)}^{\,c}\cap \overline{G}_{1,N(l_2)}^{\,c})\times (G_{2,l_1}\cap G_{2,l_2})=\overline{G}_{1,N(l_2)}^{\,c}\times G_{2,l_1}$. This allows us to define a holomorphic function $a\in H(G_1^c\times G_2)$ such that $a=a_l$ on $\overline{G}_{1,N(l)}^{\,c}\times G_{2,l}$.

To describe the action of $A$, fix $f\in H(G_1)$. Recall that
$$f(\zeta)=\frac{1}{(2\pi i)^n}\int_{\partial_0 G_{1,N(l)+1}}\frac{f(w)}{w-\zeta}dw,\ \ \zeta\in\overline{G}_{1,N(l)}. $$
Approximating the last integral using Riemann sums, we obtain
$$Af(z)=\frac{(-1)^n}{(2\pi i)^n}\int_{\partial_0 G_{1,N(l)+1}} f(w)a(w,z)dw,\ \ z\in G_{2,l}.$$
Taking the limit as $l\rightarrow\infty$, we extend $Af$ to every $z\in G_2$.

Conversely, if $a\in H(G_1^c\times G_2)$, then $a$ is holomorphic on a domain $D$ such that for every $l\in\mathbb{N}$, there exists $m(l)$ with $\overline{G}_{1,m(l)}^c\times G_{2,l}\subset D$. Fix $l\in\mathbb{N}$. For $f\in H(G_1)$ and $z\in G_{2,l}$, define
\begin{equation}\label{inverse}
Af(z):=\frac{1}{(2\pi i)^n} \int_{\partial_0 G_{1,m(l)+1}} f(w)a(w,z)dw.
\end{equation}
Observe that (\ref{inverse}) defines a holomorphic function on $G_2$.  Moreover, the operator $A$ is bounded. Indeed, if $\lVert f\rVert_{G_{1,m(l)+1}}\leq 1$, then from (\ref{inverse}) we have
\begin{equation}\label{inverse 2}
\rVert Af\lVert_{G_{2,l}}\leq \frac{1}{(2\pi)^n}\lVert a\rVert_{\overline{G}_{1,m(l)+1}^c\times G_{2,l}} \prod_{k=1}^{\infty} l(\partial_0 G^{(1)}_{k,m(l)+1}),
\end{equation}
where $l(\partial_0 G^{(1)}_{k,m(l)+1})$ denotes the length of $\partial_0 G^{(1)}_{k,m(l)+1}$. 

So far we have shown that $\Phi$ is a vector space isomorphism.  
We now prove that $\Phi^{-1}$ is in fact a topological isomorphism.  
By the Open Mapping Theorem (see \cite[p.~289]{Vogt}), it suffices to establish continuity of $\Phi^{-1}$.  

Suppose that $a_n \to 0$ in $H(G_1^b \times G_2)$. Then there exists an open set $U \supset G_1^b \times G_2$ such that $a_n \in H(U)$ for all $n$.  
By (\ref{inverse 2}), we then have
\[
\sup\{\,|\Phi^{-1}(a_n)(f)(z)| : f \in B,\, z \in K \,\} \;\longrightarrow\; 0\ \textup{ as }\ n\rightarrow \infty
\]
for every bounded set $B \subset H(G_1)$ and every compact set $K \subset G_2$,  
which completes the proof.

\end{proof}

\begin{remark}\label{uwaga}
\textup{Observe that the proof of Theorem \ref{pierwsze twierdzenie}, with almost no modification, remains valid for any domain $G_2\subset\mathbb{C}^k$ with $k\in\mathbb{N}$.} 
\end{remark}

\begin{corollary}\label{wniosek}
Let $G_1,\,G_2\subset\mathbb{C}^n$ be domains and let $r\in\mathbb{R}_{>0}^n$ be such that $G_1\subset\mathbb{P}_n(r)$. Then the map 
$$\Phi:\mathcal{L}(H(G_1),H(G_2))\longrightarrow H(\mathbb{P}_n(r)^c \times G_2)$$
defined by
$$\Phi (A)(\zeta,z):=A\left(\frac{1}{w_1-\zeta_1}\cdot\ldots\cdot\frac{1}{w_n-\zeta_n}\right)(z),$$
for $A\in\mathcal{L}(H(G_1),H(G_2))$ and $(\zeta,z)\in \mathbb{P}_n(r)^c \times G_2$, is a monomorphism between vector spaces.
\end{corollary}

\subsection{Linear operators between spaces of holomorphic functions on complete $n$-circled domains}

Let 
$$\mathbb{C}^n\owns (z_1,\ldots,z_n)\xrightarrow{R} (|z_1|,\ldots,|z_n|)\in\mathbb{R}^n_{\geq 0}.$$
A set $A\subset\mathbb{C}^n$ is called \emph{\(n\)-circled} if it satisfies $A=R^{-1}(R(A))$. Moreover, $A$ is said to be \emph{complete \(n\)-circled} if 
$$A=\bigcup_{z\in A}\mathbb{P}_n(|z_1|,\ldots,|z_n|).$$

\begin{proposition}
If $G\subset\mathbb{C}^n$ is a bounded, complete $n$-circled domain, then $H(\overline{G})$ is dense in $B(G)$. In particular, polynomials are dense in $B(G)$.
\end{proposition} 
\begin{proof}
For $\varphi\in B(G)$ define 
$$\varphi_r(z):=\varphi(rz),\ \ 0< r\leq 1.$$
Clearly, every function $\varphi_r(z)$ belongs to $H(\overline{G})$ and $\varphi_r$ converges to $\varphi$ in $B(G)$ as $r\rightarrow 1$. Since polynomials are dense in $H(\overline{G})$, it follows that they are also dense in $B(G)$.
\end{proof}

Let $G$ be a bounded complete n-circled domain and let $M$ the set of vectors $k$ whose coordinates are relatively prime nonnegative integers. For every $k\in M$ let $\eta (k)$ be a point at which the monomial $\lvert z^k\rvert$ attains a maximum on $\partial G$ (such a point may not be unique, in which case we choose one of them). The corresponding polydisc $\mathbb{P}_n\big{(}R(\eta (k))\big{)}$ is denoted simply by $\Delta(k)$.
Define the set
$$G^d:=\bigcup_{a\in M}\partial_0 \Delta(k)^c.$$
Finally, let $\Delta(0)$ denote any polydisc centered at the origin and contained in $G$.

\begin{thm}\textup{(\cite{Leeuv})}\label{Aizenberg}
For any $f\in H(G)\cap\mathcal{C}(\overline{G})$ the integral representation 
$$f(z)=\frac{1}{(2\pi i)^n}\left[\int_{\partial_0\Delta(0)}f(\zeta)\frac{d\zeta}{\zeta}+\sum_{k\in M}z^k \int_{\partial_0 \Delta(k)}\frac{f(\zeta)}{\zeta^k - z^k}\frac{d\zeta}{\zeta} \right]$$
holds for $z\in G$, where the series converges absolutely and uniformly on compact subsets of $G$.
\end{thm}

For $k\in M$ define
$$H_k(G):=\left\{f\in H(G):f(z)=\sum_{l=1}^\infty c_{l}z^{kl}\textup{ for some }(c_{l})_{l=1}^\infty\in\mathbb{C}^{\mathbb{N}}\right\}.$$
We also set
$$H_0:=\{f\in H(G): f\textup{ is constant}\}.$$
Every space $H_k(G)$ and $H_0(G)$ is endowed with the topology induced from $H(G)$. 

Next, we introduce the projections 
$$P_k:H(G)\owns \Sigma_{m\in\mathbb{N}^n} c_{m}z^{m}\mapsto\Sigma_{l=1}^\infty c_{kl}z^{kl}\in H_k(G) 
.$$

\begin{remark}\label{uwaga} Let $k\in M$. Then:
\begin{enumerate}
\item $H_k(G)$ is a closed subspace of $H(G)$.
\item For $f\in H(G)$ we have
$$(P_k f)(z)=\frac{z^k}{(2\pi i)^n} \int_{\partial_0 \Delta(k)}\frac{f(\zeta)}{\zeta^k - z^k}\frac{d\zeta}{\zeta},\ \ \ \ z\in G.$$
\end{enumerate}
\end{remark}

\begin{proposition}
Topology induced from $H(G)$ on $H_k(G)$ coincides with the topology induced from $H(\Delta(k))$.
\end{proposition}

\begin{proof}
It is clear that the topology induced from $H(G)$ is finer than that induced from $H(\Delta(k))$.  To prove the converse, fix $0<r<s<1$ . The following estimate holds for $f\in H_k(G)$
\begin{multline*}
\lVert f\rVert_{rG}\overset{\textup{Rem.}\,\ref{uwaga}}{=}\big\lVert \frac{z^k}{(2\pi i)^n}\int_{s\partial_0 \Delta (k)}\frac{f(\zeta)}{\zeta^k - z^k}\frac{d\zeta}{\zeta}\big\rVert_{r G}=\big\lVert\frac{1}{(2\pi i)^n)}\sum_{l=1}^\infty z^{kl}\int_{s\partial_0 \Delta(k)}\frac{f(\zeta)}{\zeta^{kl}}\frac{d\zeta}{\zeta}\big\lVert_{rG} \\
\leq \sum_{l=1}^\infty \big\lVert\frac{1}{(2\pi i)^n} z^{kl}\int_{s\partial_0 \Delta(k)}\frac{f(\zeta)}{\zeta^{kl}}\frac{d\zeta}{\zeta}\big\lVert_{rG}\overset{(*)}{\leq} \sum_{l=1}^\infty \big\lVert\frac{1}{(2\pi i)^n} z^{kl}\int_{s\partial_0 \Delta(k)}\frac{f(\zeta)}{\zeta^{kl}}\frac{d\zeta}{\zeta}\big\lVert_{r\Delta(k)}\leq  \lVert f\lVert_{s\Delta(k)}\sum_{l=1}^{\infty}\big(\frac{s}{r}\big)^{|k|l},
\end{multline*}
where (*) holds since $|z^k|\leq |\zeta^k|$ for $z\in G$ and $\zeta\in\Delta(k)$.
\end{proof}

\begin{proposition}\label{dalej}
$P_k$ is continuous for every $k\in M$.
\end{proposition}

\begin{proof}
By Remark \ref{uwaga} and the previous proof, for $f\in H_k(G)$ we have
\begin{equation}\label{oszacowanie projekcji}
\lVert P_kf\rVert_{r\Delta(k)}=\big{\lVert} z\mapsto\frac{z^k}{(2\pi i)^n}\int_{s\partial_0 \Delta(k)}\frac{f(\zeta)}{\zeta^k-z^k}\frac{d\zeta}{\zeta}\big{\rVert}_{r\Delta(k)}\leq \frac{r^{|k|}}{s^{|k|}-r^{|k|}}\lVert f\rVert_{s\Delta(k)},
\end{equation}
where $0<r<s<1$. This proves continuity of $P_k$. Finally, note that $P_k^2=P_k$. 
\end{proof}

As an immediate consequence, we obtain a direct connection between $H(G)$ and $\prod_{k\in M}H_k (G)$.

\begin{corollary}\label{rozkład H(G)}
Let $\left(Q_k\right)_{k\in M}\subset \prod_{k\in M}H_k(G)$. Then, $\sum_{k\in M}Q_k$ belongs to $H(G)$ if and only if, for every $0<r<1$ there exist a constant $C>0$ and $0<\theta<1$ such that $\lVert Q_k\rVert_{r G}\leq \,C\theta^{|k|}$ for all $k\in M$.
\end{corollary}

By the maximum principe, the evaluation map
$$P_0:H(G)\owns f\mapsto \left( z\mapsto f(0)\right)\in H_0(G)
$$
is also continuous.

The following representation of the identity operator in terms of $Q_l$  is of intrinsic interest and will appear again in what follows.
\begin{proposition}\label{need} The identity operator on $H(G)$ is given by the sum of all projections $P_l$, that is,
$$\mathbb{I}\textup{d}_G=P_0+\Sigma_{k\in M}P_k$$
\end{proposition}
\begin{proof}
Fix $0<r<s<1$ and choose $N>1$ such that $\left(\frac{r}{s}\right)^N\leq \frac{1}{2}$. Then, for every $f\in H(G)$ we have
\begin{multline*}
\lVert f-P_0 f-\sum_{k\in M,\,|k|< N}P_kf\rVert_{rG}=\big{\lVert}\frac{1}{(2\pi i)^n}\sum_{k\in M,\,|k|\geq N}z^k \int_{s\partial_0 \Delta(k)}\frac{f(\zeta)}{\zeta^k - z^k}\frac{d\zeta}{\zeta}\big{\rVert}  \\ 
\leq \sum_{k\in M,\,|k|\geq N}\big{\lVert}\frac{1}{(2\pi i)^n}\int_{s\partial_0 \Delta(k)}\frac{f(\zeta)}{\zeta^k - z^k}\frac{d\zeta}{\zeta}\big{\rVert} \overset{\textup{(\ref{oszacowanie projekcji})}}{\leq}\ \rVert f\rVert_{s G}\sum_{k\in M,\,|k|\geq N}\frac{\theta^{|k|}}{1-\theta^{|k|}}\leq 2\rVert f\rVert_{s G}\sum_{m=N}^\infty \binom{n+m-1}{n-1}\theta^m,
\end{multline*}
where $\theta=\frac{r}{s}$. The convergence of the final series ensures the desired approximation, completing the proof.
\end{proof}

\begin{thm}\label{thm:4.8}
Let $G_1, G_2 \subset \mathbb{C}^n$ be bounded, complete $n$-circled domains, and let 
$A \colon H(G_1) \to H(G_2)$ be a bounded linear operator.  
Let $M$ denote the index set corresponding to $G_1$.  
Then, for every compact set $K \subset G_2$ there exist a constant $C > 0$ and a family of sets 
$\{L_k\}_{k \in M}$ satisfying $L_k \Subset \Delta(k)$ and $\bigcup_{k \in M} L_k \Subset G_1$ such that 
\[
\lVert (A \circ P_k)f \rVert_{K} \leq C \lVert f \rVert_{L_k}, 
\qquad f \in H(G_1), \; k \in M.
\]

In particular, each operator $A \circ P_k$ extends uniquely to a bounded linear operator 
\[
A \circ P_k \colon H(\Delta(k)) \longrightarrow H(G_2),
\]
and, by a slight abuse of notation, we continue to denote the extension by $A \circ P_k$. 
\end{thm}


\begin{proof}
Fix $k\in M$ and set
$$A_k:=A\circ P_k:H(G_1)\rightarrow H(G_2).$$ 
Since $A$ is continuous, for every $0<t<1$ there exist $0< r<1$ and $\varepsilon>0$ such that
$$A\big{(}\{f\in H(G_1):\lVert f\rVert_{r G_1}\leq\varepsilon\}\big{)}\subset\{g\in H(G_2):\lVert g\rVert_{tG_2}\leq 1\}.$$
Choose $s\in (r,1)$. Then for any $f\in H(G_1)$ we obtain
\begin{equation}\label{rownosc}
\lVert A_k f\rVert_{t G_2}\leq\frac{1}{\varepsilon}\lVert P_k f \rVert_{r G_1} \overset{\textup{(\ref{oszacowanie projekcji})}}{\leq} \frac{r^{|k|}}{\varepsilon (s^{|k|}-r^{|k|})}\rVert  f\rVert_{s \Delta(k)}.
\end{equation}
It is therefore enough to set 
$$C:=\frac{r}{\varepsilon(s  -r)},\ \ \ \ L_k:=s\Delta(k).$$ 

It remains to establish the final part of the statement. By the Hahn-Banach Theorem $A_k$ extends to a continuous operator $\widetilde{A}_{k,r}:B(s\Delta(k))\rightarrow B(tG_2)$. 
Arguing as in the proof of Theorem \ref{pierwsze twierdzenie}, one shows that  
$$a_{k,r}(\zeta,z):=\widetilde{A}_{k,r}\left(\frac{1}{w_1-\zeta_1}\cdot\ldots\cdot\frac{1}{w_n-\zeta_n}\right) (z),\ \ \ (\zeta,z)\in s\overline{\Delta}(k)^{\,c} \times tG_2,$$
is holomorphic on $s\overline{\Delta}(k)^{\,c} \times tG_2$ and vanishes at points where at least one coordinate is at infinity. Moreover, since 
$$a_{k,r}(\zeta,z)=A_k\left(\frac{1}{w-\zeta}\right)(z),\ \ (\zeta,z)\in \overline{\mathbb{P}}_n(\alpha)^c \times tG_2,$$
where $\alpha=(\alpha_1,\ldots,\alpha_n),\ \alpha_j:=\sup\{|z_j|:\,z=(z_1,\ldots,z_n)\in G_1\}$, the function 
$$a_k:\Delta(k)^c\times G_2\rightarrow\mathbb{C},\ \ \ a_k=a_{k,r}\ \ \textup{on}\ \ s\overline{\Delta}(k)^{\,c}\times tG_{2},$$
 is holomorphic  in $\Delta(k)^{\,c} \times G_2$.

Finally, by Corollary \ref{wniosek}, $a_k$ corresponds to a unique bounded operator $\widetilde{A}_k\in\mathcal{L}(H(\Delta(k)),H(G_2))$.
\end{proof}

The converse of the preceding theorem is also true and can be stated as follows.

\begin{thm}\label{thm:4.9}
Let $G_1,\,G_2\subset\mathbb{C}^n$ be bounded, complete $n$-circled domains and let $M$ denote the index set corresponding to $G_1$. Let $(A_k\colon H(\Delta(k))\rightarrow H(G_2))_{k\in M}$ be a sequence of bounded linear operators satisfying the following condition: for every compact set $K\subset G_2$ there exist a positive constant $C$ and a family of compact sets $\{L_k\}_{k\in M}$ with $L_k\Subset\Delta(k)$ and $\bigcup_{k\in M} L_k\Subset G_1$ such that 
\[
\lVert A_k f\rVert_{K} \leq C \lVert f\rVert_{L_k}, \quad f\in H(\Delta(k)).
\] 
Then there exists a unique bounded linear operator $A\in\mathcal{L}(H(G_1),H(G_2))$ such that 
\[
A\big|_{H_k(G_1)} \equiv A_k \rvert_{H_k(G_1)} \qquad \text{for every } k\in M.
\]
\end{thm}

\begin{proof}
Define
$$Af:=\sum_{k\in M}(A_k\circ P_k )f,\ \ \ \ f\in H(G_1).$$
Fix $0<t<s<1$. Recall that for $k\in M$ we have
$$\lVert P_k f\rVert_{tG_1} \leq \frac{t^{|k|}}{s^{|k|}-t^{|k|}}\lVert f\rVert_{s\Delta(k)},\ \ f\in H(G_1).$$
By the definition of $A$, it  satisfies $A\rvert_{H_k(G_1)}\equiv (A_k)\rvert_{H_k(G_1)}$ for all $k\in M$. The boundedness of $A$ is a direct consequence of the estimate:
\begin{multline*}
\lVert Af\rVert_{r G_2}\leq\ \sum_{k\in M}\lVert A_k(P_k f)\rVert_{r G_2}\leq C\sum_{k\in M}\lVert P_k f\rVert_{t \Delta(k)}  
\leq C\sum_{k\in M}\frac{t^{|k|}}{s^{|k|}-t^{|k|}}\lVert f\rVert_{s \Delta(k)} \\
\leq C\left\{\binom{n+N-1}{n}\frac{t}{s-t}+2\sum_{m=N}^\infty\binom{m+n-1}{n-1}\big{(}\frac{t}{s}\big{)}^m\right\}\,\lVert f\rVert_{s G_1},
\end{multline*}
where $N\in\mathbb{N}$ is chosen so that $2\big{(}\frac{t}{s}\big{)}^N\leq 1$. Thus, $A\in\mathcal{L}(H(G_1),H(G_2))$. 
\end{proof}

We define
\[
\mathcal{L}_0(H(G_1),H(G_2)) \;:=\; \bigl\{\, L \in \mathcal{L}(H(G_1),H(G_2)) : L(\mathbf{1}) = 0 \,\bigr\},
\]
where $\mathbf{1}$ denotes the constant function $z \mapsto 1$ on $G_1$.

\begin{corollary}
Let $G_1, G_2 \subset \mathbb{C}^n$ be bounded, complete $n$-circled domains, and let 
$A\colon H(G_1) \to H(G_2)$ be a bounded operator. Suppose that the index set $M$ corresponding to $G_1$ is finite. Then there are natural isomorphisms of Fréchet spaces
\[
\mathcal{L}_0(H(G_1), H(G_2)) \ \cong\ 
\bigoplus_{k\in M} \mathcal{L}(H_k(G_1), H(G_2)) .
\]
\end{corollary}

\vspace{1cm}
{\bf Declarations}

{\bf Data Availability Statement.} No data was used for the research described in the article.

{\bf Conflict of interest} The author states that there is no conflict of interest.

\end{document}